\documentclass[11pt]{amsart}
\usepackage{bbm,amsthm,amsmath,amsfonts,amscd,amssymb}
\usepackage{stmaryrd}
\usepackage[all, cmtip]{xy}
\usepackage{enumerate}
\usepackage{color}
\theoremstyle{plain}

\newcommand{\printname}[1] {}

\newtheorem{theorem}{Theorem}[section]
\newtheorem{proposition}[theorem]{Proposition}
\newtheorem{lemma}[theorem]{Lemma}
\newtheorem{corollary}[theorem]{Corollary}

\newtheorem{definition}[theorem]{Definition}
\newtheorem{example}[theorem]{Example}
\newtheorem{remark}[theorem]{Remark}

\numberwithin{equation}{section}

\newcommand{\uG}[1]{\G^{(#1)}}

\newcommand{\Lie}{\mathcal{L}}          
\renewcommand{\gg}{\mathfrak{g}}        
\newcommand{\hh}{\mathfrak{h}}          
\newcommand{\X}{\mathcal{X}}            
\newcommand{\Y}{\mathcal{Y}}            
\newcommand{\G}{\mathcal{G}}
\newcommand{\toto}{\rightrightarrows}

\newcommand{\ad}{\mathrm{ad}}

\newcommand{\sour}        {\mathsf{s}}
\newcommand{\tar}         {{\mathsf{t}}}

\newcommand{\Arrow}{\rightarrow}        

\begin{document}
\title{Lie groupoids and the Fr\"olicher-Nijenhuis bracket}

\author{Henrique Bursztyn}
\address{IMPA, Estrada Dona Castorina 110, Rio de Janeiro, 22460-320, Brazil}
\email{henrique@impa.br}

\author{Thiago Drummond}
\address{Instituto de Matem\'atica, UFRJ, Av. Athos da Silveira Ramos 149, CT - Bloco C,
Rio de Janeiro, 21941-901, Brazil  }
\email{drummond@im.ufrj.br}


\thanks{\mbox{~~~~}
MSC2010 Subject Classification Number: 58HXX.
\newline
\mbox{~~~~}Keywords: Lie groupoids, Fr\"olicher-Nijenhuis bracket,
multiplicative vector-valued forms}


\begin{abstract}
The space of vector-valued forms on any manifold is a graded Lie
algebra with respect to the Fr\"olicher-Nijenhuis bracket. In this
paper we consider multiplicative vector-valued forms on Lie
groupoids and show that they naturally form a graded Lie subalgebra.
Along the way, we discuss various examples and different
characterizations of multiplicative vector-valued forms.
\end{abstract}
\maketitle

\tableofcontents

\setcounter{tocdepth}{1}


\section{Introduction}             %

Lie groupoids are ubiquitous in several areas of mathematics; they
arise as models for singular spaces, in the study of foliations and
group actions, noncommutative geometry, Poisson geometry, etc. (see
e.g. \cite{CW,CDW,Mac-book,Mo,MM} and references therein). In these
settings, one is often led to consider Lie groupoids endowed with
additional geometric structures compatible with the groupoid
operation, referred to as {\em multiplicative}. Examples of interest
include multiplicative symplectic and Poisson structures
\cite{LW,Mac-Xu,We87} (see also \cite{AC,bc,BC,bcwz,ILX}), complex
structures \cite{LSX}, and distributions \cite{CSS,Haw,JO}. The
present paper fits into the broader project of studying
multiplicative structures on Lie groupoids and should be seen as a
companion to \cite{BDK}. Here we focus on multiplicative
vector-valued forms and study their compatibility with the
Fr\"olicher-Nijenhuis bracket \cite{FN}.

There are several algebraic objects naturally associated with a
smooth manifold $M$, such as the de Rham complex
$(\Omega^\bullet(M),d)$, the Gerstenhaber algebra of multivector
fields $(\Gamma(\wedge^\bullet
TM),[\cdot,\cdot]_{\scriptscriptstyle{SN}})$, where
$[\cdot,\cdot]_{\scriptscriptstyle{SN}}$ denotes the
Schouten-Nijenhuis bracket (see e.g. \cite[Sec.~7.5]{Mac-book}), and
the graded Lie algebra of vector-valued forms
$$
(\Gamma(\wedge^\bullet T^*M\otimes
TM),[\cdot,\cdot]_{\scriptscriptstyle{FN}}),
$$
where $[\cdot,\cdot]_{\scriptscriptstyle{FN}}$ is the
Fr\"olicher-Nijenhuis bracket \cite{FN} (see e.g.
\cite[Sec.~8]{nat}). These objects play a key role in measuring the
integrability of geometric structures on $M$: for example, a
differential form on $M$ is closed if it is a cocycle in the de Rham
complex, a bivector field $\Lambda \in \Gamma(\wedge^2 TM)$ is a
Poisson structure if it satisfies
$[\Lambda,\Lambda]_{\scriptscriptstyle{SN}}=0$, and an almost
complex structure $J\in \Gamma(T^*M\otimes TM)$ is a complex
structure if $[J,J]_{\scriptscriptstyle{FN}}=0$.

When $M$ is replaced by a Lie groupoid $\G$, the relevant issue is
whether these natural algebraic operations are compatible with {\em
multiplicative} geometric structures. It is a simple verification
that multiplicative forms define a subcomplex of
$(\Omega^\bullet(\G),d)$; it is also known that the
Schouten-Nijenhuis bracket restricts to multiplicative multivector
fields, making them into a Gerstenhaber subalgebra of
$(\Gamma(\wedge^\bullet T\G),
[\cdot,\cdot]_{\scriptscriptstyle{SN}})$ \cite[Sec.~2.1]{ILX}. We
verify in this paper that an analogous result holds for
multiplicative vector-valued forms on $\G$, i.e., we show that the
space of multiplicative vector-valued forms is closed under the
Fr\"olicher-Nijenhuis bracket. Some of the applications of this
result will be discussed in \cite{BDK}.

The paper is organized as follows. In Section \ref{sec:LG} we review
the key examples of Lie groupoids that are relevant to the paper. In
Section \ref{sec:mult} we consider multiplicative vector-valued
forms on Lie groupoids and discuss examples, including relations
with connections and curvature on principal bundles. Section
\ref{Sec:FN} contains the main results: we give a direct proof of
the compatibility of multiplicative vector-valued forms and the
Fr\"olicher-Nijenhuis bracket in Thm.~\ref{thm:compatible}, and then
see how this result follows from a broader, more conceptual,
perspective, in which multiplicative vector-valued forms are
characterized in terms of the Bott-Shulman-Stasheff complex of a Lie
groupoid.

\smallskip

{\em It is a pleasure to dedicate this paper to IMPA's 60th
anniversary.}

\medskip

\noindent {\bf Acknowledgments:} We are grateful to J. Palis for his
encouragement in the preparation of these notes. We thank A. Cabrera
and N. Kieserman for useful discussions, and D. Carchedi for helpful
advice (particularly on Remark \ref{rem:TN}). H. B. was partially
supported by FAPERJ.

\section{Lie groupoids and examples} \label{sec:LG}


This section recalls some examples of Lie groupoids relevant to the
paper; further details can be found e.g. in \cite{CW,Mac-book,MM}.

Let $\G$ be a Lie groupoid over a manifold $M$, denoted by $\G\toto
M$. As usual, we refer to $\G$ as the {\em space of arrows} and $M$
as the {\em space of objects}. We denote the source and target maps
by $\sour$, $\tar: \G\to M$, the multiplication map by
$$
m: \G^{(2)}:=\{(g,h)\in \G\times \G,|\, \sour(g)=\tar(h)\} \to \G,
$$
the unit map by $\epsilon: M\to \G$, and inversion by $\iota:\G\to
\G$, $\iota(g)=g^{-1}$. We often identify $M$ with its image under
the embedding $\epsilon$ and use the notation $\epsilon(x)= 1_x$. We
also write $m(g,h)=gh$ to simplify notation. If there is any risk of
confusion, we use the groupoid itself to label its structure maps:
$\sour_\G$, $\tar_\G$, $m_\G$, $\epsilon_\G$, $\iota_\G$.

A {\it morphism} from $\G\toto M$ to $\mathcal{H} \toto N$ is a pair
of smooth maps $F:\G\to \mathcal{H}$, $f: M\to N$ that commute with
source and target maps, and preserve multiplication (this implies
that unit and inversion maps are also preserved).

A central observation to this paper is that, given a Lie groupoid
$\G\toto M$, its tangent bundle $T\G$ is naturally a Lie groupoid
over $TM$: its source and target maps are given by $T\sour_\G$,
$T\tar_\G: T\G\to TM$; for the multiplication, we notice that
$$
(T\G)^{(2)}=\{(X,Y)\in T\G\times T\G\;|\; T\sour_\G(X)=T\tar_\G(Y)\}
= T(\G^{(2)}),
$$
so we set $m_{T\G}=Tm_\G$. Similarly, the unit and inverse maps are
$T\epsilon_\G : TM \Arrow T\G$ and $T\iota_\G : T\G \Arrow T\G$.

Another important remark is that the Whitney sum $\oplus^k T\G$ (of
vector budles over $\G$) is naturally a Lie groupoid over $\oplus^k
TM$,
\begin{equation}\label{eq:sum}
\oplus^k T\G \toto \oplus^k TM,
\end{equation}
with structure maps defined componentwise.

We list some  basic examples of Lie groupoids and their tangent
bundles.


\begin{example}\label{ex:group}
A Lie groupoid over a point is a Lie group $G$, in which case its
tangent bundle $TG$ is also a Lie group. For $g, h \in G$ and $X \in
T_g G, Y \in T_h G$, the multiplication on $TG$ is given by
$$
Tm_G(X, Y) = Tr_h(X) + Tl_g(Y) \in T_{gh}G,
$$
where $r_g, l_h: G\to G$ denote right, left translations. Using the
trivialization $TG \simeq G \times \gg$ by right-translations, one
sees that
\begin{equation}\label{semi_direct}
Tm_G((g,u), (h,v)) = (gh, u + \mathrm{Ad}_g(v)).
\end{equation}
This identifies $TG$ with the Lie group $G\ltimes \gg$ obtained by
semi-direct product with respect to the adjoint action.
\end{example}

\begin{example}\label{ex:VB}
Any vector bundle $\pi: E \Arrow M$ can be naturally seen as a Lie
groupoid: source and target maps coincide with the projection $\pi$,
and the multiplication is given by addition on the fibers.
In this case, the tangent groupoid $TE$ over $TM$ is defined by the
vector bundle $T\pi: TE \Arrow TM$, known as the {\em tangent
prolongation} of $E$.
\end{example}

\begin{example}\label{ex:gauge}
Let $G$ be a Lie group, and let $\pi:P \Arrow M$ be a (right)
principal $G$-bundle. We denote the $G$-action on $P$ by $\psi:
P\times G\to P$,
$$
p \mapsto \psi_g(p),\;\;\; p\in P.
$$
The corresponding {\em gauge groupoid} $\mathcal{G}(P) \toto M$ is
defined as the orbit space of the diagonal action of $G$ on $P\times
P$; we write $\overline{(p,q)}$ for the image of
$(p,q)\in P\times P$ in $\G(P)$.
Source and target maps on $\G(P)$ are given by the composition of
the natural projections $P\times P\to P$ with $\pi$, and
multiplication is given by
$$
\overline{(p,q)} \cdot \overline{(p',q')}=\overline{(p,q')},
$$
where we assume in this composition that $q=p'$ (given any
representatives $(p,q)$ and $(p',q')$, we have that
$\pi(q)=\pi(p')$, so for a fixed $(p,q)$ one may always replace
$(p',q')$ by a unique point in its $G$-orbit satisfying the desired
property). The unit map is
$$
\epsilon: M \to \G(P), \;\; x\mapsto \overline{(p,p)},
$$
where $p\in P$ is any point such that $\pi(p)=x$, whereas the
inversion is given by
$$
\overline{(p,q)}\mapsto \overline{(q,p)}.
$$

The $G$-action on $P$ naturally induces a $TG$-action on $TP$ by
\begin{equation}\label{eq:Psi}
\Psi_{(g,u)}(X_q)= T\psi_g(X_q) + u_P(\psi_g(q)),
\end{equation}
for $X_q\in T_qP$ and $(g,u)\in TG\cong G\ltimes \gg$; here $u_P \in
\mathfrak{X}(P)$ is the infinitesimal generator of the $G$-action on
$P$. This action makes $T\pi: TP\to TM$ into a principal
$TG$-bundle, so we have a corresponding gauge groupoid $\G(TP)$. One
may verify that there is a natural identification between $\G(TP)$
and the tangent groupoid $T\mathcal{G}(P) \rightrightarrows TM$:
\begin{equation}\label{eq:TP}
T\G(P) = \G(TP).
\end{equation}
We denote the image of an element $(X,Y)\in TP\times TP$ in $\G(TP)$
by $\overline{(X,Y)}$. The induced vector bundle structure
$\G(TP)\to \G(P)$ is given by
$$
\overline{(X_1,Y_1)} + \overline{(X_2,Y_2)} =
\overline{(X_1+X_2,Y_1+Y_2)},\;\;
\lambda\overline{(X,Y)}=\overline{(\lambda X, \lambda Y)},
$$
where, for the addition, the representatives are chosen over the
same fiber of $TP\times TP \to P\times P$.
\end{example}

For a vector bundle $E\to M$ (of rank $n$), let $\mathrm{GL}(E)$ be
the gauge groupoid of the frame $\mathrm{GL}(n)$-bundle
$\mathrm{Fr}(E)\to M$. More concretely, $\mathrm{GL}(E)\toto M$ is
the Lie groupoid whose arrows between $x, y \in M$ are linear
isomorphisms from $E_x$ to $E_y$.
A {\em representation} of $\G\toto M$ on a vector bundle $E\to M$ is
a groupoid homomorphism from $\G$ into $\mathrm{GL}(E)$.

\begin{example}\label{ex:semi}
Given a representation of a Lie groupoid $\G\toto M$ on a vector
bundle $\pi: E\to M$, there is an associated {\em semi-direct
product} Lie groupoid $\G \ltimes E\toto M$: its space of arrows is
$$
\tar^* E = \G \times_{\tar,\pi} E = \{ (g,e)\,|\, \tar(g)=\pi(e)\},
$$
with source and target maps given by $(g,e)\mapsto \sour_\G(g)$ and
$(g,e)\mapsto \tar_\G(g)$, respectively, and multiplication given by
\begin{equation}\label{eq:semid}
((g_1,e_1),(g_2,e_2))\mapsto (g_1g_2, e_1 + g_1\cdot e_2),
\end{equation}
where we write $g\cdot e$ for the action $\G \times_{\sour,\pi} E
\to E$ induced by the representation.

There is an induced representation of $T\G\toto TM$ on $TE\to TM$,
and the tangent groupoid to $\G \ltimes E$ is the corresponding semi-direct product.

\end{example}

\section{Multiplicative vector-valued forms} \label{sec:mult}

\subsection{Definition and first examples}

A {\em vector-valued $k$-form} on a manifold $N$ is an element in
$\Omega^k(N,TN) := \Gamma(\wedge^kT^*N\otimes TN)$. It will be
convenient to think of vector-valued $k$-forms as maps
$$
\oplus^k TN \to TN.
$$
In particular, vector-valued 1-forms $K\in \Omega^1(N,TN)$ are
naturally identified with endomorphisms $TN\to TN$ (covering the
identity).

Given a Lie groupoid $\G\toto M$, we will be concerned with
vector-valued forms on $\G$ which are compatible with the groupoid
structure in the following sense \cite{BDK,LSX}.

\begin{definition}
A vector-valued form $K \in \Omega^k(\G, T\G)$ is {\em
multiplicative} if there exists $K_M \in \Omega^k(M, TM)$ such that
\begin{equation}\label{mult_diagram}
\xymatrix{
\oplus^k T\G \ar@<-3pt>[d] \ar@<3pt>[d] \ar[r]^{\hspace{10pt}K} & T\G \ar@<-3pt>[d] \ar@<3pt>[d]\\
\oplus^k TM \ar[r]^{\hspace{10pt}K_M} & TM\\
}
\end{equation}
is a groupoid morphism.
\end{definition}
In this case, we say that $K$ {\it covers} $K_M$.

\begin{example}
Let $G$ be a Lie group. An endomorphism $J:TG\to TG$, viewed as a
vector-valued 1-form $J\in \Omega^1(G,TG)$, is multiplicative if and
only if
$$
J\circ Tm = Tm\circ (J\times J).
$$
In particular, if $J$ is an integrable almost complex structure on
$J$, then it is multiplicative if and only if $m: G\times G \to G$
is a holomorphic map, i.e., $J$ makes $G$ into a complex Lie group
(the fact that the inversion map is holomorphic automatically
follows).

In general, a multiplicative vector-valued $k$-form on a Lie group
$G$ may be equivalently viewed as a multiplicative $k$-form on $G$
with values on the adjoint representation \footnote{Given a Lie
groupoid $\G\toto M$ along with a representation on $E\to M$, recall
from \cite[Sec.~2.1]{CSS} that a form $\omega \in
\Omega^k(\G,\tar^*E)$ is {\em multiplicative} if it satisfies
$m^*\omega |_{(g,h)} = pr_1^*\omega + g\cdot pr_2^*\omega$, where
$(g,h)\in \G^{(2)}$ and $pr_1, pr_2:\G^{(2)}\to \G$ are the natural
projections.}: to verify this fact, we use the identification $TG =
G\times \gg = \tar^*\gg$, recalling that the target map is the
trivial map $\tar: G \Arrow \{\ast\}$, and notice that, for $K \in
\Omega^k(G, TG)$, \eqref{semi_direct} implies that
\eqref{mult_diagram} is a Lie groupoid morphism if and only if
\begin{equation}\label{eq:adj}
(m^*K)_{(g,h)} = pr_1^*K + \mathrm{Ad}_g(pr_2^*K), \,\, \, \text{
for } g, \, h \in G,
\end{equation}
where $pr_1$, $pr_2: G\times G\to G$ are the natural projections.
\end{example}

A Lie groupoid $\G\toto M$ is called {\em holomorphic} if it is
equipped with a complex structure $J\in \Omega^1(\G,T\G)$ that is
multiplicative. Besides complex Lie groups, holomorphic vector
bundles provide natural examples:

\begin{example}
Let $(M,J_M)$ be a complex manifold and consider a (real) vector
bundle $\pi: E \Arrow M$, viewed as a Lie groupoid as in
Example~\ref{ex:VB}. A vector-valued $k$-form $K\in \Omega^k(E,TE)$
is multiplicative in this case if and only if the associated map
$\oplus^kTE\to TE$ is a vector-bundle morphism with respect to the
vector-bundle structures $\oplus^kTE\to \oplus^kTM$ and $TE\to TM$.
It is observed in \cite{LSX} that an integrable almost complex
structure $J \in \Omega^1(E,TE)$ which is multiplicative and covers
$J_M \in \Omega^1(M, TM)$ is equivalent to equipping $E$ with the
structure of a holomorphic vector bundle over $M$.
\end{example}


Other examples of multiplicative vector-valued forms arise in the
context of connections on principal bundles, as we now discuss.

\subsection{Principal connections and curvature}

Let $G$ be a Lie group and $\pi: P\to M$ be a principal (right)
$G$-bundle. We will follow the notation of Example~\ref{ex:gauge}.

Let $V\subseteq TP$ be the vertical bundle over $P$, i.e., the fiber
of $V\to P$ over $p\in P$ is
$$
V_p = \{u_P(p)\,|\, u\in \gg\},
$$
where $u_P \in \mathfrak{X}(P)$ is the infinitesimal generator of
the $G$-action on $P$. The vertical bundle $V\to P$ induces a
distribution\footnote{We always assume distributions to be of
constant rank, i.e., subbundles of the tangent bundle.}
$$
\Delta_V \subseteq T\G(P)
$$
on the gauge groupoid $\G(P)$ given by the image of $V\times
V\subset TP\times TP$ under the quotient map $TP\times TP\to \G(TP)
= T\G(P)$.

\subsubsection{Principal connections} \

Let $\theta\in \Omega^1(P,\gg)$  be a principal connection on $P$.
By using the identification (of $G$-equivariant
vector bundles over $P$)
$$
P\times \gg \to V, \;\;\; (p,u)\mapsto u_P(p),
$$
we may equivalently describe it as a $G$-equivariant 1-form
$\Theta\in \Omega^1(P,V)$ such that
\begin{equation}\label{eq:Theta}
\mathrm{Im}(\Theta)=V,\;\;  \mbox{ and }\; \Theta^2= \Theta,
\end{equation}
so that $\Theta(X)= (\theta(X))_P$. We denote the horizontal bundle
defined by the connection by $H:=\ker(\theta) =\ker(\Theta)
\subseteq TP$.


We observe that principal connections on $P$ are naturally
associated with certain multiplicative vector-valued 1-forms on
$\G(P)$:

\begin{proposition}\label{prop:correspondence}
There is a one-to-one correspondence between principal connections
$\theta \in \Omega^1(P,\gg)$ on $P$ and multiplicative $K\in
\Omega^1(\G(P),T\G(P))$ satisfying $K^2=K$ and
$\mathrm{Im}(K)=\Delta_V$.
\end{proposition}

For the proof, we need some general observations.

\begin{lemma}\label{lem:distr}
Let $D \subset TP$ be a $G$-invariant distribution on $P$. Let $\hh
\subseteq \gg$ be an $\mathrm{Ad}$-invariant subspace, and suppose
that
$$
D\cap V |_p = \{u_P(p)\;|\; u\in \hh\},
$$
at each $p\in P$. Then the image $\Delta_D$ of $D\times D \subset
TP\times TP$ under the quotient map $TP\times TP \to T\G(P)$ is
distribution which is a Lie subgroupoid of $T\G(P)$:
$$
(\Delta_D\toto D_M) \hookrightarrow (T\G(P)\toto TM),
$$
for $D_M=T\pi(D)\subset TM$.
\end{lemma}

\begin{proof}
One may directly check from \eqref{eq:Psi} that $D\subseteq TP$ is
$G\ltimes \hh$-invariant, where we view the semi-direct product Lie
group $G\ltimes \hh\subseteq TG$ as a subgroup of $TG$.

For $X$, $Y\in D$, $T\pi(X)=T\pi(Y)$ if and only if
$$
X=\Psi_{(g,u)}Y = T\psi_g(Y) + u_P,
$$
but since $D$ is $G$-invariant, it follows that $u_P\in D$, hence
$u\in \hh$. It follows that $T\pi(X)=T\pi(Y)$ if and only if $X$ and
$Y$ are on the same $G\ltimes \hh$-orbit.

Since $V\cap D$ has constant rank, $\Delta_D$ is a subbundle of
$T\G(P)$, and $D_M= T\pi(D)$ is a subbundle of $TM$. To verify that
$\Delta_D\toto D_M$ is Lie subgroupoid of $T\G(P)$, let
$$
\overline{(X_i,Y_i)} \in \Delta_D|_{\overline{(p_i,q_i)}},
$$
for $i=1,2$, be composable, i.e., $T\pi(Y_1)=T\pi(X_2)$. For $Y_1\in
D|_{q_1}$ and $X_2\in D|_{p_2}$, we saw that this implies the
existence of $(g,u) \in G\ltimes \hh$ such that $Y_1 =
\Psi_{(g,u)}(X_2)$, where $q_1=\psi_g(p_2)$. Hence
\begin{align*}
Tm(\overline{(X_1, Y_1)}, \overline{(X_2, Y_2)} ) &=
Tm(\overline{(X_1,Y_1)},\overline{(\Psi_{(g,u)}(X_2),\Psi_{(g,u)}(Y_2))} )\\
& =\overline{(X_1, \Psi_{(g,u)}(Y_2))},
\end{align*}
which belongs to $\Delta_D$ since $D$ is $G\times\hh$-invariant.
\end{proof}

The following are particular instances of Lemma~\ref{lem:distr}:

\begin{itemize}
\item The vertical bundle $V$ satisfies the conditions of
Lemma~\ref{lem:distr} with $\hh=\gg$. In this case $V_M=T\pi(V)=M$,
and we have a corresponding subgroupoid
$$
\Delta_V\toto M
$$
of $T\G(P)\toto TM$.

\item For $D=H$ the horizontal bundle of a principal connection, the
conditions in Lemma~\ref{lem:distr} hold for $\hh=\{0\}$; then
$H_M=T\pi(H)=TM$, and we have a subgroupoid
$$
\Delta_H\toto TM.
$$

\end{itemize}

For an arbitrary Lie groupoid $\G\toto M$, recall that a
distribution $\Delta \subset T\G$ is called {\em multiplicative} if
it is a Lie subgroupoid of $T\G \toto TM$. In this case, the space of objects
of $\Delta$ is a subbundle $\Delta_M\subseteq TM$ (see e.g. \cite{jotz12}).


\begin{lemma}\label{lem:distrib1}
Let $\Delta^1$, $\Delta^2$ be distributions on $\G$ satisfying $T\G
= \Delta^1 \oplus \Delta^2$. If $K\in \Omega^1(\G, T\G)$ is a
projection so that $\Delta^1= \mathrm{Im}(K)$ and $\Delta^2 =
\ker(K)$, then $K$ is multiplicative if and only if both $\Delta^1$
and $\Delta^2$ are multiplicative distributions.
\end{lemma}


\begin{proof}
Suppose that $K$ is multiplicative, i.e., a groupoid morphism
$T\G\to T\G$. The fact that $\Delta^2=\ker(K)$ is multiplicative
follows from the more general fact that the kernel of morphisms of
VB-groupoids (see e.g. \cite[Ch.~11]{Mac-book}) is a VB-subgroupoid
whenever it has constant rank, see e.g. \cite{BCH,Li-Bland}. The
analogous result for $\Delta^1$ follows since $\mathrm{Id}-K$ is
also a multiplicative projection and $\Delta^1$ is its kernel.

To prove the converse, note that the spaces of units of the grupoids
$\Delta^1$ and $\Delta^2$, denoted by $\Delta_M^1$ and $\Delta_M^2$,
are subbundles of $TM$ satisfying $TM=\Delta_M^1 \oplus \Delta_M^2$. Let $K_M: TM \Arrow TM$
be the projection on
$\Delta^1_M$ along $\Delta^2_M$. It is clear that $K$ and $K_M$
intertwine the source and target maps for $T\G \toto TM$. For
$$
X = X_1 + X_2 \in \Delta^1 \oplus \Delta^2 \,\,\,\,\text{ and }
\,\,\,\, Y = Y_1 + Y_2 \in \Delta^1 \oplus \Delta^2
$$
satisfying $T\sour(X) = T\tar(Y)$, we see that
$T\sour(X_1)=T\tar(Y_1)$, $T\sour(X_2)= T\tar(Y_2)$ and
$$
\begin{array}{rl}
K(Tm(X,Y))= K(Tm(X_1, Y_1) + Tm(X_2, Y_2)) =&  K(Tm(X_1, Y_1))\\
  = & Tm(X_1, Y_1)\\
  = & Tm(K(X), K(Y)).
\end{array}
$$
So $K$ preserves groupoid multiplication.
\end{proof}


We can now prove Prop.~\ref{prop:correspondence}.

\begin{proof}

Consider a connection on $P$ given by $\Theta\in \Omega^1(P,V)$. Let
$K: T\G(P)\to T\G(P)$ be defined by
\begin{equation}\label{eq:KT}
K(\overline{(X,Y)})=\overline{(\Theta(X),\Theta(Y))}.
\end{equation}
The properties of $\Theta$ (see \eqref{eq:Theta}) imply that $K$ is
well defined (by the $G$-equivariance of $\Theta$), satisfies $K^2=K$, and
that $\mathrm{Im}(K)=\Delta_V$
and $\mathrm{Ker}(K)=\Delta_H$. By Lemmas~\ref{lem:distr} and
\ref{lem:distrib1}, $K$ is multiplicative.

Conversely, let $K$ be a multiplicative vector-valued 1-form
satisfying $K^2=K$ and $\mathrm{Im}(K) = \Delta_V$. Let us consider
the vector bundle $TP/G \to M$, and its subbundle $V/G\to M$. We
note that $K$ naturally induces a projection map
\begin{equation}\label{eq:eqproj}
\bar{\Theta}: TP/G \to V/G
\end{equation}
as follows. First recall that there is a natural identification of
$TP/G$ with $\ker(T\sour_{\G(P)})|_M$ as vector bundles over $M$:
indeed, noticing that
$$
\ker(T\sour_{\G(P)})|_x =\{ \overline{(X_p,Y_p)}, \; T\pi(Y)=0\},
$$
where $\pi(p)=x\in M$, the identification $TP/G \to
\ker(T\sour_{\G(P)})|_M$ is given by
$$
\overline{X}|_{\pi(p)} \mapsto \overline{(X_p,0_p)},
$$
where $\overline{X}$ denotes the class of $X \in TP$ in $TP/G$. The
inverse map is $\overline{(X_p,Y_p)} \mapsto \overline{X_p-Y_p} \in
(TP/G)|_{\pi(p)}$. Under this identification, the subbundle $V/G
\subset TP/G$ corresponds to $\Delta_V |_M \subset
\ker(T\sour_{\G(P)})|_M$. The projection map \eqref{eq:eqproj} is
defined by the diagram
\begin{equation}\label{eq:diag}
\xymatrix{
 TP/G
 \ar[r]^-{\sim}\ar[d]_-{\bar{\Theta}} & \ker(T\sour_{\G(P)})|_M \ar[d]^-{K} \\
V/G \ar[r]_-{\sim} & \Delta_V|_M. }
\end{equation}
The map $\bar{\Theta}$ is equivalent to a connection $\Theta\in
\Omega^1(P,V)$ through
$\bar{\Theta}(\overline{X})=\overline{\Theta(X)}$. This is the
connection defined by $K$.

More explicitly, the relation between $\Theta$ and $K$ in diagram
\eqref{eq:diag} is
\begin{equation}\label{eq:Ktheta}
K(\overline{(X_p,0_p)})=\overline{(\Theta(X_p),0_p)},
\end{equation}
and, as we now see, this condition completely determines $K$: Using
the groupoid structure on $T\G(P)$, we can write an arbitrary
$\overline{(X_p,Y_q)}$ as
$$
\overline{(X_p,Y_q)}=\overline{(X_p,0_p)} \cdot
\overline{(0_p,0_q)}\cdot \overline{(0_q,Y_q)} =
\overline{(X_p,0_p)} \cdot \overline{(0_p,0_q)} \cdot
\overline{(Y_q,0_q)}^{-1},
$$
and, since $K$ is multiplicative, \eqref{eq:Ktheta} implies that
$$
K(\overline{(X_p,Y_q)})= \overline{(\Theta(X_p),0_p)}\cdot
\overline{(0_p,0_q)} \cdot \overline{(\Theta(Y_q),0_q)}^{-1}=
\overline{(\Theta(X_p),\Theta(Y_q))}.
$$

It follows (see \eqref{eq:KT}) that the construction relating $K$
and $\Theta$ just described are inverses of one another.
\end{proof}

\subsubsection{Curvature}\

For a manifold $N$, the \textit{curvature} of a projection $K:
TN\to TN$ is the vector-valued 2-form $R_K \in \Omega^2(N,TN)$ given
by
\begin{equation}\label{eq:pcurv}
R_K(X,Y)= K([(\mathrm{Id} - K)(X), (\mathrm{Id} - K)(X)]), \;\;
X,Y\in \mathfrak{X}(N),
\end{equation}
where $[\cdot,\cdot]$ is the Lie bracket of vector fields (see e.g. \cite{nat}).
So $R_K$ measures the integrability of the
distribution $\mathrm{Ker}(K)\subseteq TN$. The
\textit{co-curvature} of $K$ is the curvature of $\mathrm{Id}-K$.

A direct consequence of the results in Section~\ref{Sec:FN} (see
Theorem~\ref{thm:compatible}) is that, on a Lie groupoid, the
curvature of any multiplicative projection is a multiplicative
vector-valued 2-form. We will now verify this fact in the case of
projections on gauge groupoids $\G(P)$ arising from principal
connections, as in Prop.~\ref{prop:correspondence}. In this
particular context, the result follows from the explicit relation
between $R_K \in \Omega^2(\G(P),T\G(P))$ and the curvature of the
connection corresponding to $K$, as explained in
Prop.~\ref{prop:curv} below.

Let $\mathrm{Ad}(P)\to M$ be the vector bundle associated with the
adjoint action on $\gg$, i.e., $\mathrm{Ad}(P)=(P\times \gg)/G$. We
denote elements in $\mathrm{Ad}(P)$ by $\overline{(p,v)}$, for $p
\in P$ and $v \in \gg$. There is a natural representation of the
gauge groupoid $\G(P)$ on $\mathrm{Ad}(P)$ by
\begin{equation}\label{eq:adjact}
\overline{(q,p)} \cdot \overline{(p, v)} = \overline{(q, v)}.
\end{equation}
As in Example~\ref{ex:semi}, we consider the semi-direct product
groupoid $\G(P)\ltimes \mathrm{Ad}(P)$, that we denote by
$\tar^*\mathrm{Ad}(P)\toto M$.

\begin{lemma}\label{lem:isom} The following holds:
\begin{itemize}
\item[(a)] There is a natural groupoid isomorphism
\begin{equation}\label{eq:phi}
\varphi: (\Delta_V\toto M) \to (\tar^*\mathrm{Ad}(P)\toto M),
\end{equation}
which is also a isomorphism of vector bundles over $\G(P)$. (I.e.,
this is a isomorphism of VB-groupoids \cite[Ch.~11]{Mac-book}.)
\item[(b)] Assume that a vector-valued $k$-form $R \in \Omega^k(\G(P),T\G(P))$
takes values in $\Delta_V\subseteq T\G(P)$. Then $R$ is
multiplicative if and only if $R' := \varphi \circ R \in
\Omega^{k}(\G(P), t^*\mathrm{Ad}(P))$ satisfies
\begin{equation}\label{mult_rep}
(m^*R')_{(g,h)} = pr_1^* R' + g\cdot pr_2^*R',
\end{equation}
for $(g,h)\in \G(P)^{(2)}$. (I.e., $R'$ is multiplicative as a
$k$-form with values on the representation $\mathrm{Ad}(P)$, as in
\cite{CSS}.)
\end{itemize}
\end{lemma}

\begin{proof}
We define the map $\varphi: \Delta_V \to \tar^*\mathrm{Ad}(P)$ by
$\varphi(\overline{(u_P(p),v_P(q))}) = \overline{(p, u-v)}$. One can
directly verify that this map is well-defined, and that it is a
morphisms of vector bundles over $\G(P)$; the inverse map
$\tar^*\mathrm{Ad}(P) \to \Delta_V$ is defined, on each fiber over
$\overline{(p,q)}\in \G(P)$, by $\overline{(p,v)} \mapsto
\overline{(v_P(p),0_P(q))}$. To verify that $\varphi$ is a groupoid
morphism, fix $g= \overline{(p_1, q_1)}$, $h= \overline{(p_2, q_2)}
\in \G(P)$, and
$$
\X= \overline{(u^1_P(p_1), v^1_P(q_1))} \in
\Delta_V|_{\overline{(p_1,q_1)}},\;\;\; \Y= \overline{(u^2_P(p_2),
v^2_P(q_2))} \in \Delta_V|_{\overline{(p_2,q_2)}}.
$$
Since $T\tar(\Y)= T\sour(\X)$, we can assume that $q_1=p_2$ and
$v^1=u^2$. So
$$
\varphi(Tm(\X, \Y)) = \varphi( \overline{(u^1_P(p_1), v^2_P(q_2))} =
\overline{(p_1, u^1- v^2)}.
$$
On the other hand,
\begin{equation*}
\begin{split}
\varphi(\X) + g\cdot \varphi(\Y) & =  \overline{(p_1, u^1-v^1)} +
\overline{(p_1, q_1)}
\cdot \overline{(p_2, u^2-v^2)}\\
& = \overline{(p_1, u^1- v^1)} + \overline{(p_1, v^1 - v^2)} =
\overline{(p_1, u^1 - v^2)},
\end{split}
\end{equation*}
hence multiplication is preserved (c.f. Example~\ref{ex:semi}).

The claim in part (b) follows directly from (a) (and
\eqref{eq:semid}).
\end{proof}

\begin{remark}
The observation in Lemma~\ref{lem:isom}, part (a), is an instance of
a more general fact: on any regular Lie groupoid $\G\toto M$, there
is a natural representation of $\G$ on the vector subbundle
$\ker(\rho) \subset A$, where $A$ is the Lie algebroid of $\G$ and
$\rho$ is its anchor; in this case the distribution
$\ker(T\sour)\cap \ker(T\tar)\subseteq T\G$ is multiplicative, and
naturally isomorphic to the semi-direct product groupoid $\G \ltimes
\ker(\rho)$ (as a groupoid and as a vector bundle over $\G$). When
$\G$ is a gauge groupoid $\G(P)$, $\ker(\rho)=\mathrm{Ad}(P)$, and
$\ker(T\sour)\cap \ker(T\tar) = \Delta_V$.
\end{remark}

For a connection $\theta\in \Omega^1(P,\gg)$, let $H \subset TP$ be
its horizontal bundle. For a vector field $X \in \mathfrak{X}(P)$,
let $X^H$ be its projection on $H$: $X^H = (\mathrm{Id}-\Theta)(X)$.
Let $F_\theta \in \Omega^2(P,\gg)$ be the curvature of $\theta$,
$$
F_\theta(X,Y) = -\theta([X^H,Y^H]).
$$
Since $F_\theta$ is invariant and $i_XF_\theta=0$ for $X\in V$, it
may be alternatively viewed as an element in
$\Omega^2(M,\mathrm{Ad}(P))$. Let $K\in \Omega^1(\G(P),T\G(P))$ be
the projection corresponding to $\theta$, and let $R_K\in
\Omega^2(\G(P), T\G(P))$ be its curvature \eqref{eq:pcurv}. Using
\eqref{eq:phi}, we consider
$$
R_K' = \varphi\circ R_K \in \Omega^2(\G(P), \tar^*\mathrm{Ad}(P)).
$$

\begin{lemma}\label{lem:multR}
 $R_K'$ satisfies
\begin{equation}\label{eq:cohtrivial}
R_K'|_g =  g \cdot
 (\sour^*F_\theta) - \tar^*F_\theta,
\end{equation}
where $F_\theta \in \Omega^2(M,
\ad(P))$ is the curvature of $\theta$.
\end{lemma}

\begin{proof}
Let us fix $g=\overline{(p, q)}\in \G(P)$, $\X = \overline{(X_1,
Y_1)}$, $\Y= \overline{(X_2, Y_2)} \in T\G(P)|_{g}$,
for $X_1, X_2 \in T_pP$ and $Y_1, Y_2 \in T_qP$. By definition (see
\eqref{eq:pcurv}),
$$
R_K(\X, \Y) = K\left( \overline{( [X_1^H, X_2^H](p), [Y_1^H,
Y_2^H](q))}\right) = \overline{( \Theta([X_1^H, X_2^H](p)), \Theta(
[Y_1^H, Y_2^H](q)) )},
$$
where $X_i^H, Y_i^H \in \mathfrak{X}(P)$ are horizontal vector
fields extending $(\mathrm{Id} - \Theta)(X_i)$ and $(\mathrm{Id} -
\Theta)(Y_i)$, respectively, for $i=1,2$. Hence,
$$
\varphi (R_K(\X, \Y) ) = \overline{(p, \theta([X_1^H, X_2^H](p)) -
\theta([Y_1^H, Y_2^H](q)) )}.
$$
On the other hand,
$$
F_\theta(T\tar(\X),T\tar(\Y)) = F_\theta(T\pi(X_1), T\pi(X_2)) =
-\overline{(p, \theta([X_1^H, X_2^H](p)))},
$$
and
$$
\overline{(p, q)} \cdot F_\theta(T\pi(Y_1), T\pi(Y_2)) = -
\overline{(p, q)} \cdot \overline{(q, \theta([Y_1^H, Y_2^H]) )} = -
\overline{(p, \theta([Y_1^H, Y_2^H]) )}.
$$
\end{proof}

\begin{proposition}\label{prop:curv}
If $K\in \Omega^1(\G(P),T\G(P))$ is a projection corresponding to a
connection on $\G(P)$, then $R_K\in \Omega^2(\G(P),T\G(P))$ is
multiplicative.
\end{proposition}

\begin{proof}
Using Lemmas~\ref{lem:isom} and \ref{lem:multR}, the result follows
once we check that condition \eqref{eq:cohtrivial} implies that
\eqref{mult_rep} holds. Considering the maps $m, pr_1, pr_2:
\G(P)^{(2)}\to \G(P)$, this can be directly verified using the
identities $\tar\circ m = \tar\circ pr_1$, $\sour\circ m =
\sour\circ pr_2$, and $\tar\circ pr_2 = \sour\circ pr_1$.
\end{proof}

\begin{remark}
In the context of multiplicative forms on Lie groupoids with
coefficients on representations \cite{CSS}, Lemma~\ref{lem:multR}
may be interpreted as the fact that $R_K'$ is ``exact'', or
``cohomologically trivial'', while the weaker condition
\eqref{mult_rep}, that guarantees multiplicativity, corresponds to
``closedness'' (see \cite[Sec.~2.1 and Sec.~3.4]{CSS}).
\end{remark}

\section{The Fr\"olicher-Nijenhuis bracket}\label{Sec:FN}

Let $N$ be a manifold and $\Omega^\bullet(N)$ be its graded algebra
of differential forms. A {\it degree $l$ derivation} of
$\Omega^\bullet(N)$ is a linear map $D: \Omega^\bullet(N)\to
\Omega^{\bullet + l}(N)$ such that $D(\alpha \wedge\beta)=
D(\alpha)\wedge \beta + (-1)^{pl}\alpha\wedge D(\beta)$, for $\alpha
\in \Omega^{p}(N)$. Any vector-valued form $K \in \Omega^k(N,TN)$
gives rise to a degree $(k-1)$ derivation of $\Omega^\bullet(N)$ by
\begin{align*}
& i_K\omega(X_1, \dots, X_{k+p-1}) =  \\
& \frac{1}{k!(p-1)!}\sum_{\sigma \in
S_{k+p-1}} sgn(\sigma) \,\omega(K(X_{\sigma(1)},
\dots, X_{\sigma(k)}), X_{\sigma(k+1)}, \dots, X_{\sigma(k+p-1)}),
\end{align*}
for $\omega \in \Omega^p(N)$, $X_1, \ldots, X_{k+p-1} \in TN$. It
also gives rise to a degree $k$ derivation of $\Omega^\bullet(N)$
via
$$
 \Lie_{K} = [d, i_K] = di_K -(-1)^{k-1}i_K d,
$$
where $d$ is the exterior differential on $N$.

Given $K\in \Omega^k(N,TN)$ and $L\in \Omega^l(N,TN)$, their {\it
Fr\"olicher-Nijenhuis bracket} is the vector-valued form $[K,L]\in
\Omega^{k+l}(N,TN)$ uniquely defined by the condition
\begin{equation}\label{eq:comm}
\Lie_{[K, L]} = [\Lie_K, \Lie_L] = \Lie_K \Lie_L - (-1)^{kl} \Lie_L
\Lie_K.
\end{equation}
When $K$ and $L$ have degree zero (i.e., they are vector fields on
$N$), \eqref{eq:comm} agrees with the usual Lie bracket of vector
fields. The Fr\"olicher-Nijenhuis bracket makes
$\Omega^\bullet(N,TN)$ into a graded Lie algebra,
and it satisfies the following additional properties (see e.g.
\cite[Ch.~2]{nat}):
\begin{itemize}
\item[(a)] For $K\in \Omega^1(N,TN)$,
\begin{equation}\label{eq:Nij}
\frac{1}{2}[K,K] = N_K,
\end{equation}
where $N_K$ is the Nijenhuis tensor of $K$,
$$
N_K(X,Y)=[K(X),K(Y)]-K([KX,Y]+ [KY,X]) + K^2[X,Y],
$$
for $X,Y \in TN$.

\item[(b)] When $K\in \Omega^1(N,TN)$ is a projection, then
\begin{equation}\label{eq:RR}
\frac{1}{2}[K,K] = R_K + \overline{R}_K,
\end{equation}
where $R_K$ is its curvature and $\overline{R}_K$ is its co-curvature.

\item[(c)] Let $f: N_1\to N_2$ be a smooth map, and $K_i\in \Omega^k(N_i,TN_i)$,
$L_i\in \Omega^l(N_i,TN_i)$, $i=1,2$, be such that $K_1$ is
$f$-related to $K_2$ and $L_1$ is $f$-related to $L_2$. Then
$[K_1,L_1]$ is $f$-related to $[K_2,L_2]$.

\end{itemize}

Regarding property (c), recall that $K_1 \in \Omega^k(N_1,TN_1)$ is
{\it $f$-related} to $K_2 \in \Omega^k(N_2,TN_2)$ if
$$
K_2(Tf(X_1),\ldots,Tf(X_k))=Tf(K_1(X_1,\ldots,X_k)),
$$
for all $X_1,\ldots,X_k \in T_xN$, and $x\in N$. Alternatively,
$K_1$ and $K_2$ are $f$-related if and only if
$$
\Lie_{K_1} \circ f^* = f^* \circ \Lie_{K_2}
$$
where $f^*: \Omega(N_2) \Arrow \Omega(N_1)$ is the pull-back of
differential forms. We refer to the property in (c) above as the
{\it naturality} of the Fr\"olicher-Nijenhuis bracket.

\subsection{The bracket on Lie groupoids}

We now verify that the Fr\"olicher-Nijenhuis bracket on a Lie
groupoid $\G\toto M$ preserves multiplicative vector-valued forms.

We start by giving an alternative characterization of multiplicative
vector-valued forms. We say that $K \in \Omega^k(\G, T\G)$ is
\textit{$(\sour,\tar)$-projectable} if there exists $K_M \in
\Omega^k(M, TM)$ such that $K$ is both $\sour$ and $\tar$-related to
$K_M$.

Any $K \in \Omega^k(\G, T\G)$ gives rise to a
vector valued $k$-form $K\times K$ on $\G\times \G$ given by
$$
K\times K((X_1, Y_1), \dots, (X_k, Y_k)) = (K(X_1, \dots, X_k), K(Y_1, \dots, Y_k)),
$$
for $X_1, \dots, X_k \in T_g\G$ and $Y_1, \dots, Y_k \in T_h\G$;
this form is uniquely characterized by the fact that it is both
$pr_1$ and $pr_2$-related to $K$, where $pr_1, pr_2: \G\times \G
\Arrow \G$ are the natural projections.

\begin{lemma}\label{mult_lemma}
If $K$ is $(\sour,\tar)$-projectable, then $K\times K$ restricts to a vector
valued $k$-form $K^{(2)}$ on the space of composable arrows
$\G^{(2)}$. Moreover, $K$ is multiplicative if and only if $K$ is
$(\sour,\tar)$-projectable and $K^{(2)}$ is $m$-related to $K$.
\end{lemma}

\begin{proof}
A direct computation shows that $K\times K$ restricts to $\G^{(2)}$ when $K$ is
$(\sour,\tar)$-projectable.

Recall that $K$ is multiplicative if and only if there exists $K_M
\in \Omega^k(M,TM)$ such that \eqref{mult_diagram} is a groupoid
morphism. The existence of $K_M$ is equivalent to $K$ being
$(\sour,\tar)$-projectable, whereas the identity
$$
\begin{array}{rl}
K(Tm(X_1, Y_1), \dots, Tm(X_k, Y_k))  = & Tm(K(X_1, \dots, X_k), K(Y_1, \dots, Y_k)\\
                                      = & Tm(K^{(2)}((X_1, Y_1), \dots, (X_k,
                                      Y_k))),
\end{array}
$$
for $(X_1, Y_1), \dots, (X_k, Y_k) \in T_{(g,h)} \G^{(2)}$, shows
that $K$ intertwines the multiplication if and only if $K$ and
$K^{(2)}$ are $m$-related.
\end{proof}

We also need the following observation:

\begin{lemma}\label{2_frolicher}
If $K$ and $L$ are $(\sour,\tar)$-projectable, then
$$
[K,L]^{(2)} = [K^{(2)}, L^{(2)}].
$$
\end{lemma}

\begin{proof}
Since $K^{(2)}$ (resp. $L^{(2)}$) is $pr_1$ and $pr_2$-related to
$K$ (resp. $L$), it follows from the naturality of the
Fr\"olicher-Nijenjuis bracket that $[K^{(2)}, L^{(2)}]$ is both
$pr_1$ and $pr_2$-related to $[K,L]$. Since $[K, L]^{(2)}$ is the
unique vector-valued form satisfying this property, it follows that
$[K,L]^{(2)} = [K^{(2)}, L^{(2)}]$.
\end{proof}

\begin{theorem} \label{thm:compatible}
Let $K\in \Omega^k(\G,T\G)$ and $L \in
\Omega^l(\G,T\G)$ be multiplicative. Then  $[K,L]$ is multiplicative.
\end{theorem}

\begin{proof}
By naturality, $[K,L]$ is $\sour$ and $\tar$-related to $[K_M,
L_M]$. So $[K, L]$ is $(\sour,\tar)$-projectable. Similarly, since
$K^{(2)}$ (resp. $L^{(2)}$) and $K$ (resp. $L$) are $m$-related, it
follows that $[K^{(2)},L^{(2)}]$ and $[K, L]$ are $m$-related. By
Lemma \ref{2_frolicher},  $[K,L]^{(2)}$ and $[K,L]$ are $m$-related.
The result now follows from Lemma \ref{mult_lemma}.
\end{proof}

\begin{corollary}
The following holds:
\begin{itemize}
 \item[(a)] If $K\in \Omega^1(\G,T\G)$ is multiplicative, then its Nijenhuis
 tensor $N_K\in \Omega^2(\G,T\G)$ is multiplicative.
 \item[(b)] If $K\in \Omega^1(\G,T\G)$ is a multiplicative projection, then its curvature
$R_K\in \Omega^2(\G,T\G)$ is multiplicative.
\end{itemize}
\end{corollary}

\begin{proof}
Part (a) follows directly from \eqref{eq:Nij}.

From \eqref{eq:Nij} and \eqref{eq:RR}, we see that, when $K$ is a
projection, $R_K(X,Y)= K(N_K(X,Y))$ for $X,Y\in T\G$, and this
proves part (b), since both $K$ and $N_K$ are multiplicative, i.e.,
groupoid morphisms (and hence so is their composition).
\end{proof}

Note that part (a) recovers \cite[Prop.~3.3]{LSX}; Part (b)
generalizes Prop.~\ref{prop:curv}.

\subsection{Relation with the Bott-Shulman-Stasheff complex}

In this final section, we provide an alternative characterization of
multiplicative vector-valued forms leading to another viewpoint to
Thm.~\ref{thm:compatible}.

For a smooth manifold $N$, as previously mentioned, the
Fr\"olicher-Nijenhuis bracket makes $\Omega^\bullet(N,TN)$ into a
graded Lie algebra. This is a consequence of fact that the map
\begin{equation}\label{eq:map}
K \mapsto \Lie_{K}
\end{equation}
identifies vector-valued forms on $N$ with derivations of
$\Omega^\bullet(N)$ commuting (always in the graded sense) with the
exterior differential \cite{FN}, which is itself a graded Lie
algebra with respect to commutators. The Fr\"olicher-Nijenhuis
bracket is the unique bracket on $\Omega^\bullet(N,TN)$ for which
\eqref{eq:map} is an isomorphism of graded Lie algebras (see
\eqref{eq:comm}). We will show that this result extends to
multiplicative vector-valued forms on Lie groupoids.

For a Lie groupoid $\G\toto M$, let us consider the associated
simplicial manifold $N(\G)$, known as its {\it nerve}, defined as
follows: for each $p\in \mathbb{N}$, its $p$ component is
$\G^{(p)}$, the string of $p$ composable arrows (i.e.,
$(g_1,\ldots,g_p) \in \G^p$ satisfying $\sour(g_{i+1})=\tar(g_i)$);
its face maps are $\partial_i^{p-1}: \G^{(p)} \Arrow \G^{(p-1)}$,
$i=0, \ldots, p$, given by
$$
\partial_i^{p-1}(g_1, \dots, g_p) =
\begin{cases}
(g_2, \dots, g_p), & \text{ if } i=0,\\
(g_1, \dots, g_{i-1}, g_{i}g_{i+1},  g_{i+2}, \dots, g_{p}), & \text{ if } 1 \leq i \leq p-1,\\
(g_1, \dots, g_{p-1}), & \text{ if } i=p,
\end{cases}
$$
for $p\geq 1$, and $\partial^0_0=\sour$, $\partial^0_1=\tar$, for
$p=1$;
the degeneracy maps $s_i^p: \G^{(p-1)} \Arrow \G^{(p)}$, $i=0,
\ldots, p-1$, are defined by
\begin{align*}
s_i^p(g_1, \dots, g_{p-1}) &= (g_1, \dots, g_i, 1_{t(g_{i+1})},
g_{i+1}, \dots, g_{p-1})\\& = (g_1, \dots, g_i, 1_{s(g_{i})},
g_{i+1}, \dots, g_{p-1}).
\end{align*}
For convenience, we recall the identities relating the face and
degeneracy maps:
\begin{equation}\label{face_deg}
\partial_j^{p-1}\circ s_i^p =
 \begin{cases}
s_{i-1}^{p-1}\circ \partial_j^{p-2}, & \text{ if $j < i$},\vspace{5pt}\\
\mathrm{Id}_{\G^{(p-1)}}, & \text{ if $j=i, i+1$},\vspace{5pt}\\
s_i^{p-1} \circ \partial_{j-1}^{p-2}, & \text{ if $j > i$}.
\end{cases}
\end{equation}

We consider the associated double complex
$\Omega^\bullet(\G^{(\bullet)})$, referred to as the {\it
Bott-Shulman-Stasheff} complex, with differentials given by the
exterior derivative $d: \Omega^q(\G^{(\bullet)})\to
\Omega^{q+1}(\G^{(\bullet)})$ and by
$$
\delta: \Omega^\bullet(\G^{(p-1)}) \Arrow
\Omega^\bullet(\G^{(p)}),\;\;\; \delta = \sum_{i=0}^{p}(-1)^{i}
(\partial_i^p)^*,
$$
and whose total cohomology (also known as the {\em de Rham
cohomology} of $\G$) agrees with the cohomology of the geometric
realization of $N(\G)$ \cite{Bott} (see also \cite{AC} and
references therein).

A {\em degree $l$ derivation} of $\Omega^{\bullet}(\G^{(\bullet)})$
is a sequence $\overline{D}=(D_0, D_1, \dots)$, where each $D_p$ is
a degree $l$-derivation of $\Omega^{\bullet}(\G^{(p)})$ and
\begin{equation}\label{face_rel}
(s_i^p)^*\circ D_p = D_{p-1} \circ (s_i^p)^*,
\end{equation}
for all $p$ and $i = 0,\ldots, p-1$.
The componentwise commutator turns the space of derivations of
$\Omega^{\bullet}(\G^{(\bullet)})$ into a graded Lie algebra. The
subspace of derivations of $\Omega^{\bullet}(\G^{(\bullet)})$
commuting with its total differential is a graded Lie subalgebra.

\begin{proposition}\label{FN_corresp}
There is a (graded) linear isomorphism between the space of
multiplicative vector-valued forms on $\G$ and the space of
derivations of $\Omega^{\bullet}(\G^{(\bullet)})$ commuting with the
total differential. Explicitly, the map taking a multiplicative $K
\in \Omega^k(\G, T\G)$ to a degree $k$ derivation is given by
\begin{equation}\label{D_K}
K \mapsto (\Lie_{K_M}, \Lie_{K}, \dots, \Lie_{K^{(p)}}, \dots),
\end{equation}
where $K^{(p)}$ is the restriction of $(K\times \dots \times K) \in
\Omega^k(\G^p,T\G^p)$ to $\G^{(p)}$.
\end{proposition}

For the Lie groupoid $N \toto N$, with $\sour=\tar=\mathrm{id}_N$,
Proposition \ref{FN_corresp} boils down to the correspondence (c.f.
\eqref{eq:map}) between vector-valued forms on $N$ and derivations
of $\Omega^\bullet(N)$ commuting with $d$. Note also that
Thm.~\ref{thm:compatible} is a consequence of
Prop.~\ref{FN_corresp}: the map \eqref{D_K} induces a graded Lie
bracket on multiplicative vector-valued forms on $\G$ which is
nothing but the restriction of the Fr\"olicher-Nijenhuis bracket.

\begin{proof}
For a multiplicative $K \in \Omega^k(\G, T\G)$, one may directly
verify that the restrictions $K^{(p)}$ are well defined (see
Lemma~\ref{mult_lemma}), and that the right-hand side of \eqref{D_K}
satisfies \eqref{face_rel} and commutes with the total differential.

Let $\overline{D}$ be a degree-$k$ derivation  of $\Omega^\bullet
(\G^{(\bullet)})$. The fact that $\overline{D}$ commutes with the
total differential gives the conditions
\begin{align}
\label{del:comm}D_p \circ \delta & = \delta\circ D_{p-1}\\
\label{deRham:comm} D_p \circ d & = (-1)^k d\circ  D_p,
\end{align}
for each $p$. Condition \eqref{deRham:comm} is simply that
$[D_p,d]=0$, so it implies that there exists $K_p \in
\Omega^{k}(\uG{p})$ such that $D_p = \Lie_{K_p}$ (see
\cite[Sec.~8.5]{nat}). So we are left with proving that $K=K_1$ is
multiplicative, covers $K_M=K_0$, and $K_p = K^{(p)}$.

From \eqref{face_rel} we see that $K_p$ and $K_{p-1}$ are
$s_i^p$-related, for $i=0, \dots, p-1$. In particular, $K$ and $K_M$
are $\epsilon$-related.

When $p=1$, \eqref{del:comm} and the fact that $K_2$ and $K$ are
$s^2_0$-related imply that
\begin{equation}\label{p=1}
\Lie_{K}\circ(s^2_0)^*\circ \delta  = (s^2_0)^*\circ \Lie_{K_2}
\circ \delta = (-1)^k (s^2_0)^*\circ\delta  \circ \Lie_K.
\end{equation}
Since $(s^2_0)^*\circ \delta = (\epsilon \circ \tar)^*$, from
\eqref{face_deg} one gets
$$
\Lie_K \circ \tar^* \circ \epsilon^* = \tar^*\circ \epsilon^* \circ
\Lie_K = \tar^* \circ \Lie_{K_M} \circ \epsilon^*.
$$
The fact that $\epsilon$ is an immersion implies that $\Lie_K \circ
\tar^* = \tar^* \circ \Lie_{K_M}$, thus proving that $K$ and $K_{M}$
are $\tar$-related. To obtain the analogous result for the source
map, it suffices to apply $(s^2_1)^*$ to \eqref{del:comm}. This
proves that $K$ covers $K_M$.

To verify the compatibility of $K$ with the multiplication on $\G$,
one needs to work at the level $p=2$. In this case, \eqref{del:comm}
reads
\begin{equation}\label{p=2}
\Lie_{K_3} \circ \delta = \delta \circ \Lie_{K_2}.
\end{equation}
By applying $(s_0^3)^*$ to \eqref{p=2} and using that $K_3$ and
$K_2$ are $s_0^3$-related, one gets
$$
\Lie_{K_2} \circ (s_0^3)^* \circ \delta = (s_0^3)^*\circ \delta
\circ \Lie_{K_2}.
$$
The identities \eqref{face_deg} imply that $(s_0^3)^*\circ \delta =
((\partial_0^1)^*-\delta) \circ (s_0^2)^*$. Hence
\begin{equation*}
\begin{split}
\Lie_{K_2}\circ (\partial^1_0)^*\circ (s_0^2)^* - \Lie_{K_2} \circ \delta \circ (s_0^2)^* = & (\partial^1_0)^*\circ (s_0^2)^*\circ \Lie_{K_2} - \delta \circ (s_0^2)^* \circ \Lie_{K_2}\\
= & (\partial_0^1)^* \circ \Lie_K \circ (s^2_0)^* - \delta \circ \Lie_K \circ (s_0^2)^*,\\
\end{split}
\end{equation*}
which implies that $\Lie_{K^{(2)}}\circ (\partial^1_0)^*\circ
(s_0^2)^* = (\partial_0^1)^* \circ \Lie_K \circ (s^2_0)^*$. Since
$s^2_0$ is an immersion, we conclude that $K_2$ and $K$ are
$\partial^1_0$-related. Arguing similarly with $s^3_2$, one obtains
that $K_2$ and $K$ are $\partial^1_2$-related. These two facts imply
that $K_2 = K^{(2)} = (K\times K)|_{\uG{2}}$. Moreover, as
$m=\partial^1_1$, one has that
$$
m^* \circ \Lie_{K} = ((\partial_0^1)^* + (\partial_2^1)^* - \delta)
\circ \Lie_{K} = \Lie_{K^{(2)}} \circ ((\partial_0^1)^* +
(\partial_2^1)^* - \delta) = \Lie_{K^{(2)}} \circ m^*,
$$
which proves that $K^{(2)}$ and $K$ are $m$-related. This concludes
the proof that $K$ is multiplicative.

To prove that $K_p = K^{(p)}$, one proceeds by induction. Assume
that $K_{p-1}=K^{(p-1)}$ and use the identity $(s_0^{p+1})^*\circ
\delta = ((\partial^{p-1}_0)^* - \delta)\circ (s_0^{p})^*$ to prove
as above that $K_p$ and $K^{(p-1)}$ are $\partial_0^{p-1}$-related.
Arguing similarly with $s_{p}^{p+1}$ proves that $K_p$ and
$K^{(p-1)}$ are $\partial^{p-1}_{p}$-related, which implies that
$K_p=K^{(p)}$.
\end{proof}

\begin{remark}\label{rem:TN}
Along the lines of the proof of Proposition \ref{FN_corresp}, one
can prove that the relations
$$
\left\{
\begin{array}{l}
(s_i^{p})^* \circ\Lie_{K_p} = \Lie_{K_{p-1}} \circ (s_i^p)^*, \, i=0, \dots, p-1\\
\Lie_{K_p} \circ \delta = \delta \circ \Lie_{K_{p-1}}
\end{array}
\right.
$$
are equivalent to each $K_p$ being both $s_i^p$- and
$\partial^{p-1}_j$-related to $K_{p-1}$, for $i=0,\dots, p-1$ and
$j=0, \dots, p$.
When $K$ is a vector field, this indicates that $({K_M},K,
K^{(2)},\ldots)$ may be thought of as a ``vector field'' on the
simplicial manifold $N(\G)$, in the sense that it defines a section
of the natural projection $T(N\G) \Arrow N(\G)$ (where $T(N(\G))$ is
the simplicial manifold obtained by taking the tangent functor on
each component of $N(\G)$). We refer to \cite{Hep} for a related
approach to vector fields on differentiable stacks. It would be
interesting to extend this picture to higher degrees.
\end{remark}



\begin{thebibliography}{99}

\bibitem{AC}
Arias Abad, C., Crainic, M.: The Weil algebra and the Van Est
isomorphism. {\em Ann. Inst. Fourier (Grenoble)} {\bf 61} (2011),
927 -- 970.


\bibitem{Bott}
Bott, R., Shulman, H., Stasheff, J.: On the de Rham theory of
certain classifying spaces. {\em Advances in Math.} {\bf 20} (1976),
43--56.

\bibitem{bc}
Bursztyn, H., Cabrera, A.: Multiplicative forms at the infinitesimal
level. {\em Math. Ann.} {\bf 353} (2012), 663--705.

\bibitem{BCH}
Bursztyn, H., Cabrera, A., del Hoyo, M.: Vector bundles over Lie groupoids and algebroids. {\em  Adv. Math.} {\bf 290} (2016), 163--207.


\bibitem{BC}
Bursztyn, H., Crainic, M.: Dirac geometry, quasi-Poisson actions and
D/G-valued moment maps. {\em J. Differential Geometry}   {\bf 82}
(2009), 501--566.

\bibitem{bcwz}
Bursztyn, H., Crainic, M., Weinstein, A., Zhu, C.: Integration of
twisted Dirac brackets, {\em Duke Math. J.} {\bf 123} (2004),
549--607.

\bibitem{BDK}
Bursztyn, H., Drummond, T.: Lie theory of multiplicative tensors. 

\bibitem{CW}
Cannas da Silva, A., Weinstein, A.: {\em Geometric models for
noncommutative algebras}. Berkeley Mathematics Lecture Notes, 10.
American Mathematical Society, Providence, RI; Berkeley Center for
Pure and Applied Mathematics, Berkeley, CA, 1999.

%

\bibitem{CDW}
Coste, A., Dazord, P., Weinstein, A.: \textit{Groupo\"ides
symplectiques}. Publications du D\'epartement de Math\'ematiques.
Nouvelle S\'erie. A, Vol. 2,  i--ii, 1--62, Publ. D\'ep. Math.
Nouvelle S\'er. A, 87-2, Univ. Claude-Bernard, Lyon, 1987.




\bibitem{CSS}
Crainic, M., Salazar, M. A., Struchiner, I.: Multiplicative forms
and Spencer operators. {\em Math. Z.}  {\bf 279} (2014), 939--979.




\bibitem{FN}
Fr\"olicher, A., Nijenhuis, A.: Theory of vector valued differential
forms. Part I. {\em Indag. Math.} {\bf 18} (1956), 338–-359.






\bibitem{Haw}
Hawkins, E. : A groupoid approach to quantization. {\em J.
Symplectic Geom.} {\bf 6} (2008), 61–-125.

\bibitem{Hep}
Hepworth, R.: {\em Vector fields and flows on differentiable
stacks.} Theory and Applications of Categories {\bf 22}(21) (2009)
542-587.


\bibitem{ILX}
Iglesias Ponte, D., Laurent-Gangoux, C., Xu, P.: Universal lifting
theorem and quasi-Poisson groupoids. {\em J. European Math. Soc.}
{\bf 14} (2012), 681-–731.

\bibitem{jotz12}
Jotz, M. : The leafspace of a multiplicative foliation. {\em J. Geom. Mech.} {\bf 3} (2012), 313--332.

\bibitem{JO}
Jotz, M., Ortiz, C.:
Foliated groupoids and infinitesimal ideal systems. {\em Indag. Math. (N.S.)} {\bf 25} (2014), 1019--1053.



\bibitem{nat}
Kol\'ar, I.; Michor, P.; Slov\'ak, J.: Natural operations in
differential geometry, {\em Springer-Verlag, Berlin Heidelberg
1993.}

\bibitem{LSX}
Laurent-Gengoux, C., Sti\' enon, M. , Xu, P.: Integration of
holomorphic Lie algebroids. {\em Math.Ann.} {\bf 345} (2009)
895--923.

\bibitem{Li-Bland}
Li-Bland, D., Meinrenken, E.: Dirac Lie groups.  {\em Asian J. of Math.}, {\bf 18} (2014) 779-–815.

\bibitem{LW}
Lu, J.-H., Weinstein, A.: Poisson Lie groups, dressing
transformations, and Bruhat decompositions. {\em J. Differential
Geom.} {\bf 31} (1990) 501--526.




\bibitem{Mac-book}
Mackenzie, K., {\em General theory of Lie groupoids and Lie
algebroids}. London Mathematical Society Lecture Note Series, 213.
Cambridge University Press, Cambridge, 2005.

\bibitem{Mac-Xu}
Mackenzie, K., Xu, P., Lie bialgebroids and Poisson groupoids.  {\em
Duke Math. J.}  {\bf 73}  (1994), 415--452


\bibitem{Mo}
Moerdijk, I.: Orbifolds as groupoids: an introduction. {\em Contemp.
Math.} {\bf 310} (2002), 205--222.

\bibitem{MM}
Moerdjik, I., Mcrun, J.: {\em Introduction to Foliations and Lie
Groupoids.} Cambridge Studies in Advanced Mathematics 91, Cambridge
University Press, Cambridge, 2003.







\bibitem {We87}
{Weinstein, A.}, {Symplectic groupoids and {P}oisson
  manifolds.}
{\em Bull. Amer. Math. Soc. (N.S.)}  {\bf 16} (1987), 101--104.





\end{thebibliography}
\end{document}